\documentclass[reqno, 11pt]{amsart}
\usepackage{amsmath}
 \usepackage{amssymb}
\usepackage{amsthm}
\usepackage{times}
\usepackage{latexsym}
\usepackage[mathscr]{eucal}

\numberwithin{equation}{section}
 
  \newtheorem{theorem}{Theorem}[section]
  \newtheorem{proposition}[theorem]{Proposition}
  
  \newtheorem{corollary}[theorem]{Corollary}

  \newtheorem{definition}[theorem]{Definition}
  \newtheorem{example}[theorem]{Example}

\title[Affine Szab\'o connections on smooth manifolds]{Affine Szab\'o connections on smooth manifolds}

\author[Abdoul Salam Diallo, Fortun\'{e} Massamba ]{Abdoul Salam Diallo*, Fortun\'{e} Massamba**}

\newcommand{\acr}{\newline\indent}

\address{\llap{*\,} School of Mathematics, Statistics and Computer Science\acr
 University of KwaZulu-Natal\acr
 Private Bag X01, Scottsville 3209\acr
South Africa  \acr
and \acr
Universit\'e Alioune Diop de Bambey\acr
UFR SATIC, D\'epartement de Math\'ematiques\acr
B. P. 30, Bambey, S\'en\'egal}
\email{Diallo@ukzn.ac.za, abdoulsalam.diallo@uadb.edu.sn}
\address{\llap{**\,} School of Mathematics, Statistics and Computer Science\acr
 University of KwaZulu-Natal\acr
 Private Bag X01, Scottsville 3209\acr
South Africa}
\email{massfort@yahoo.fr, Massamba@ukzn.ac.za}

\thanks{} 



\begin{document}
 
\begin{abstract}  
In this paper, we introduce a new structure, namely, affine Szab\'o connection. We prove that, on $2$-dimensional affine manifolds, the affine Szab\'o structure is equivalent to one of the cyclic parallelism of the Ricci tensor. A characterization for locally homogeneous affine Szab\'o surface is obtained. Examples of two- and three-dimensional affine Szab\'o manifolds are also given.
\end{abstract}

\maketitle

\section{Introduction} 

The theory of connections is a classical topic in differential geometry. It was initially developed to solve pure geometrical problems. It provides an extremely important tool to study geometrical structures on manifolds and, as such, has been applied with great sources in many different setting. For affine connections, a survey of the development of the theory can be found in \cite{ns} and references therein. In \cite{op}, Opozda classified locally homogeneous torsion-free affine connections on $2$-manifolds. Arias-Marco and Kowalski in \cite{ak} classified locally homogeneous connections with arbitrary torsion on $2$-manifolds. In  \cite{gar}, Garc\'ia-Rio \textit{et al.} introduced the notion of the affine Osserman connections. Affine Osserman connections are well-understood in dimension two. For instance, in \cite{di} and \cite{gar}, the authors proved in a different way that an affine connection is Osserman if and only if its Ricci tensor is skew-symmetric. The situation is however more involved in higher dimensions where the skew-symmetry of the Ricci tensor is a necessary (but not a sufficient) condition for an affine connection to be Osserman.

A (pseudo-) Riemannian manifold $(M,g)$ is said to be \textit{Szab\'o } if the eigenvalues of the Szab\'o operator given by
$$
\mathcal{S} (X):Y \rightarrow (\nabla_X \mathcal{R})(Y,X)X
$$ 
are constant on the unit (pseudo-) sphere bundle \cite{gis}. The Szab\'o operator is a self-adjoint operator with $\mathcal{S}(X)X =0$. It plays an important role in the study of totally isotropic manifolds \cite{giz}. Szab\'o in \cite{sz1} used techniques from algebraic topology to show, in the Riemannian setting, that any such a metric is locally symmetric. He used this observation to give a simple proof that any two point homogeneous space is either flat or is a rank one symmetric space. Subsequently Gilkey and Stavrov \cite{gs} extended his results to show that any Szab\'o Lorentzian manifold has constant sectional curvature. However, for metrics of higher signature the situation is different. Indeed it was showed in \cite{giz} the existence of Szab\'o pseudo-Riemannian manifolds with metrics of signature $(p, q)$ with $p\geq 2$ and $q\geq 2$ which are not locally symmetric. 

The aim of this paper is to extend the definition of (pseudo-) Riemannian Szab\'o manifold to the affine case by introducing a new concept called \textit{affine Szab\'o manifold}. We investigate the torsion-free affine connections to be Szab\'o. We shall call a connection with such a condition
an \textit{affine Szab\'o connection}. 

The paper is organized as follows. In Section \ref{Prel}, we recall some basic definitions and geometric objects, namely, torsion, curvature and Ricci tensor on an affine manifold. In Section \ref{Affine}, we study the cyclic parallelism of the Ricci tensor for a particular case of affine connections in two and three dimensional affine manifolds. We establish geometric configurations of affine manifolds admitting a cyclic parallel Ricci tensor (Proposition \ref{RiccCylPA1} and \ref{RiccCylPA2}). We introduce in Section \ref{AffineMani} a new concept of affine Szab\'o  manifold. We prove that, on a two-dimensional smooth affine manifold, the affine  Szab\'o structure coincides with the cyclic parallelism of the Ricci tensor (Theorem \ref{main}). We end the section by giving some examples of affine Szab\'o connections in  two and three dimensional affine manifolds. In section 5, a characterization of locally homogeneous affine Szab\'o surfaces is given. We investigate in section 6, the twisted Riemannian extension of an affine Szab\'o connection on a two-dimensional affine manifold $(M,\nabla)$. We show that the twisted Riemannian extension of an affine Szab\'o manifold is a pseudo-Riemannian nilpotent Szab\'o manifold of neutral signature and the degree of nilpotency of the Szab\'o operators depends on the direction of the unit vectors.

\section{Preliminaries}\label{Prel}

Let $M$ be a $n$-dimensional smooth manifold and $\nabla$ be an affine connection on $M$. Let us consider a system of coordinates $(u_1,u_2,\cdots, u_n)$ in a neighborhood $\mathcal{U}$ of a point $p$ in $M$. In $\mathcal{U}$, the connection is given by
\begin{equation}\label{CoefCon1}
\nabla_{\partial_i} \partial_j = f^{k}_{ij}\partial_k,
\end{equation}
where $\{\partial_i = \frac{\partial}{\partial u_{i}}\}_{1\le i\le n}$ is a basis of the tangent space $T_{p}M$ and the functions $f^{k}_{ij}(i,j,k=1,2,3,\cdots, n)$ are called the  \textit{coefficients of the affine connection}. The pair $(M, \nabla)$ shall be called \textit{affine manifold}.

Next, we define a few tensors fields associated to a given affine connection $\nabla$. The \textit{torsion tensor field} $T^{\nabla}$ is defined by
\begin{equation}
 T^{\nabla} (X,Y) = \nabla_X Y - \nabla_Y X - [X,Y],
\end{equation} 
for any vector fields $X$ and $Y$ on $M$. The components of the torsion tensor $T^{\nabla}$ in local coordinates are
\begin{equation}
 T^{k}_{ij} = f^{k}_{ij} - f^{k}_{ji}.
\end{equation}
If the torsion tensor of a given affine connection $\nabla$ vanishes, we say that $\nabla$ is torsion-free.

The \textit{curvature tensor field} $\mathcal{R}^{\nabla}$ is defined by  
\begin{equation}
 \mathcal{R}^{\nabla}(X,Y)Z:=\nabla_X \nabla_Y Z -\nabla_Y \nabla_XZ -\nabla_{[X,Y]} Z,
\end{equation} 
for any vector fields $X$, $Y$ and $Z$ on $M$. The components in local coordinates are
\begin{equation}
 \mathcal{R}^{\nabla}(\partial_k,\partial_l)\partial_j = \sum_i R^{i}_{jkl} \partial_i.
\end{equation} 
We shall assume that $\nabla$ is torsion-free. If $\mathcal{R}^{\nabla} = 0$ on $M$, we say that $\nabla$ is \textit{flat affine connection}. It is know that $\nabla$ is flat if and only if around point there exist a local coordinates system such that $f^{k}_{ij}=0$, for all $i,j$ and $k$.

We define the \textit{Ricci tensor} $Ric^{\nabla}$  by
\begin{equation}
Ric^{\nabla}(X, Y) = \mbox{trace}\{Z\mapsto \mathcal{R}^{\nabla}(Z,X)Y\}.
\end{equation}
The components in local coordinates are given by
\begin{equation}
Ric^{\nabla}(\partial_j,\partial_k) = \sum_i R^{i}_{kij}.
\end{equation}
It is known in Riemannian geometry that the Levi-Civita connection of a Riemannian metric has symmetric Ricci tensor, that is, $Ric(Y,Z) = Ric(Z,Y)$. But this property is not true for an arbitrary affine connection with torsion-free. In fact, the property is closely related to the concept of parallel volume element (see \cite{ns} for more details).

In $2$-dimensional manifold $M$, the curvature tensor $\mathcal{R}^{\nabla}$ and the Ricci tensor $Ric^{\nabla}$ are related by
\begin{equation}\label{CurRicc1}
 \mathcal{R}^{\nabla}(X, Y)Z = Ric^{\nabla}(Y, Z)X - Ric^{\nabla}(X, Z)Y.
\end{equation} 
The covariant derivative of the curvature tensor $\mathcal{R}^{\nabla}$ is given by 
\begin{equation}
 (\nabla_{X}\mathcal{R}^{\nabla})(Y, Z)W = (\nabla_{X} Ric^{\nabla})(Z, W)Y -  (\nabla_{X} Ric^{\nabla})(Y, W)Z,
\end{equation}
where the covariant derivative of the Ricci tensor $Ric^{\nabla}$ is defined as
\begin{align}\label{CoDerRicNa1}
 (\nabla_{X} Ric^{\nabla})(Z, W) & = X(Ric^{\nabla}(Z, W)) - Ric^{\nabla}(\nabla_{X}Z, W)\nonumber\\
 &- Ric^{\nabla}(Z,\nabla_{X} W).
\end{align} 
For $X\in\Gamma(T_{p} M)$, we define the affine Szab\'o operator $\mathcal{S}^{\nabla}(X)$ with respect to $X$ by $\mathcal{S}^{\nabla}(X): T_{p} M\longrightarrow  T_{p} M$ such that
\begin{equation}
 \mathcal{S}^{\nabla}(X)Y = (\nabla_{X}\mathcal{R}^{\nabla})(Y, X)X,
\end{equation}
for any vector field $Y$. The affine Szab\'o operator satisfies $\mathcal{S}^{\nabla}(X)X=0$ and $\mathcal{S}^{\nabla}(\beta X)= \beta^{3}\mathcal{S}^{\nabla}(X)$, for $\beta\in\mathbb{R}-\{0\}$ and $X\in T_{p} M$. If $Y= \partial_{m}$, for $m=1,2,\cdots,n$ and $X=\sum_{i}\alpha_{i}\partial_{i}$, one gets
\begin{equation} 
 \mathcal{S}^{\nabla}(X)\partial_{m}  = \sum_{i,j,k=1}^{n} \alpha_{i}\alpha_{j}\alpha_{k}(\nabla_{i}\mathcal{R}^{\nabla})(\partial_{m}, \partial_{j})\partial_{k},
\end{equation}
where $\nabla_{i}=\nabla_{\partial_{i}}$.

Let $A= (a_{ij})$ be the $(n\times n)$-matrix associated with the affine Szab\'o operator $\mathcal{S}^{\nabla}(X)$. Then its characteristic polynomial $P_{\lambda} [\mathcal{S}^{\nabla}(X)]$ is given by
\begin{equation}
P_{\lambda} [\mathcal{S}^{\nabla}(X)]=  \lambda^{n}-\sigma_{1}\lambda^{n-1} + \sigma_{2}\lambda^{n-2}-\cdots + (-1)^{n}\sigma_{n},
\end{equation}  
where the coefficients $\sigma_{1},\cdots,\sigma_{n}$ are given by
\begin{align} 
\sigma_{1} & =\sum_{i=1}^{n}a_{ii}= trace\, A,\nonumber\\
\sigma_{2} & =\sum_{i<j}
\left|\begin{array}{cc}
 a_{ii}& a_{ij}\\
 a_{ji}& a_{jj}
\end{array}\right|,\nonumber\\
\sigma_{3} & = \sum_{i<j<k}
\left|\begin{array}{ccc}
 a_{ii}& a_{ij}& a_{ik}\\
 a_{ji}& a_{jj}& a_{jk}\\
 a_{ki}& a_{kj}& a_{kk}
\end{array}\right|,\nonumber\\
{}&\;\;\vdots\nonumber\\
\sigma_{n}& =\det A. \nonumber
\end{align}

\section{Affine connections with cyclic parallel Ricci tensor}\label{Affine}

In this section we investigate affine connections whose Ricci tensor are cyclic parallel. Two cases of dimensions (two and three) of smooth manifolds will be considered with specific affine connections. 

We start with the formal definition.
\begin{definition}\cite{ks} {\rm An affine manifold $(M, \nabla)$ is said to be an $L_{3}$-space if its Ricci tensor $Ric^{\nabla}$ is cyclic parallel, that is 
\begin{equation}\label{RicciCycP1}
 (\nabla_X Ric^{\nabla})(X,X) = 0,
\end{equation}
for any vector field $X$ tangent to $M$ or, equivalently, if 
$$
 \mathfrak{G}_{X, Y, Z}(\nabla_{X}Ric)(Y, Z)=0,
$$
for any vector fields $X$, $Y$, $Z$ tangent to $M$, where $\mathfrak{G}_{X, Y, Z}$ denotes the cyclic sum with respect to $X$, $Y$ and $Z$.
 }
\end{definition}
Locally, the equation (\ref{RicciCycP1}) takes the form
\begin{equation}
 (\nabla_{(i}Ric^{\nabla})_{jk)}=0,
\end{equation}
or written out without the symmetrizing brackets
\begin{equation}
 (\nabla_{i}Ric^{\nabla})_{jk} + (\nabla_{j}Ric^{\nabla})_{ki} + (\nabla_{k}Ric^{\nabla})_{ij} =0.
\end{equation}
For $\displaystyle X=\sum_{i}\alpha_{i}\partial_{i}$, it is easy to show that  
\begin{equation}\label{RicciCycP2}
 (\nabla_X Ric^{\nabla})(X,X) = \sum_{i,j,k} \alpha_{i}\alpha_{j}\alpha_{k}(\nabla_{i} Ric^{\nabla})_{jk}.
\end{equation} 
Now, we are going to present two cases of affine connections in which we investigate the cyclic parallelism of the Ricci tensor.   

\textit{Case 1: }
Let $M$ be a two-dimensional smooth manifold and $\nabla$ be an affine torsion-free connection. By (\ref{CoefCon1}), we have  
\begin{equation}\label{CoefCon2}
\nabla_{\partial_i} \partial_j = f^{k}_{ij}\partial_k, \;\;\mbox{for}\;\; i,j,k=1,2,
\end{equation}
where $f_{ij}^{k} = f_{ij}^{k}(u_1,u_2)$. The components of the curvature tensor $\mathcal{R}^{\nabla}$ are given by
\begin{equation}
\mathcal{R}^{\nabla} (\partial_1,\partial_2)\partial_1  = a \partial_1 + b \partial_2,\;\;
\mathcal{R}^{\nabla} (\partial_1,\partial_2)\partial_2 =  c \partial_1 + d \partial_2,\nonumber
\end{equation}
where $a$, $b$, $c$ and $d$ are given by
\begin{align}\label{Smalabcd}
a &= \partial_1 f^{1}_{12} - \partial_2 f^{1}_{11} + f^{1}_{12}f^{2}_{12} -f^{2}_{11}f^{1}_{22},\nonumber\\
b &=  \partial_1 f^{2}_{12} - \partial_2 f^{2}_{11} + f^{2}_{11}f^{1}_{12} + f^{2}_{12}f^{2}_{12} 
- f^{1}_{11}f^{2}_{12} - f^{2}_{11}f^{2}_{22},\nonumber\\
c &=  \partial_1 f^{1}_{22} - \partial_2 f^{1}_{12}
+ f^{1}_{11}f^{1}_{22} + f^{1}_{12}f^{2}_{22} -f^{1}_{12}f^{1}_{12} - f^{2}_{12}f^{1}_{22},\nonumber\\
d &=  \partial_1 f^{2}_{22} - \partial_2 f^{2}_{12} + f^{2}_{11}f^{1}_{22} -f^{1}_{12}f^{2}_{12}.
\end{align}
From (\ref{CurRicc1}), the components of the Ricci tensor are given by
\begin{align}
 & Ric^{\nabla} (\partial_1,\partial_1) = - b,\;\; Ric^{\nabla} (\partial_1,\partial_2)= -d,\nonumber\\
 & Ric^{\nabla} (\partial_2,\partial_1)=a, \;\; Ric^{\nabla} (\partial_2,\partial_2) =c.
\end{align} 
\begin{proposition}\label{RiccCylPA1}
The affine connection $\nabla$ defined in (\ref{CoefCon2}) satisfies  (\ref{RicciCycP1}) if the functions $f^{k}_{ij}$, for $i, j, k=1,2$, satisfy the following partial differential equations:
\begin{align}
 & \partial_1 b - 2bf^{1}_{11}  - (d-a)f^{2}_{11} = 0,\nonumber\\
 & \partial_2 c -2cf^{2}_{22} + (d-a)f^{1}_{22} = 0,\nonumber\\
 & \partial_1 a - \partial_2 b - \partial_1 d + 4bf^{1}_{12} -2cf^{2}_{11} + (d-a)(f^{1}_{11} + 2 f^{2}_{12}) =0,\nonumber\\
 &\partial_2 a + \partial_1 c - \partial_2 d +2bf^{1}_{22} -4cf^{2}_{12} + (d-a)(2f^{1}_{12} + f^{2}_{22}) =0.
\end{align}
\end{proposition}
\begin{proof}
Using (\ref{RicciCycP1}) and (\ref{RicciCycP2}), one obtains 
\begin{eqnarray*}
0&=& \alpha^{3}_{1} (\nabla_{\partial_1} Ric^{\nabla})(\partial_1,\partial_1) 
+ \alpha^{3}_{2} (\nabla_{\partial_2} Ric^{\nabla})(\partial_2,\partial_2)\\
&+& \alpha^{2}_{1}\alpha_{2} \Big[(\nabla_{\partial_1} Ric^{\nabla})(\partial_1,\partial_2)
+  (\nabla_{\partial_1} Ric^{\nabla})(\partial_2,\partial_1)
+ (\nabla_{\partial_2} Ric^{\nabla})(\partial_1,\partial_1)\Big]\\ 
&+&  \alpha_{1}\alpha^{2}_{2} \Big[(\nabla_{\partial_1} Ric^{\nabla})(\partial_2,\partial_2)
+ (\nabla_{\partial_2} Ric^{\nabla})(\partial_1,\partial_2)
+ (\nabla_{\partial_2} Ric^{\nabla})(\partial_2,\partial_1)\Big].
\end{eqnarray*}
From a straightforward calculation using (\ref{CoDerRicNa1}), the components of the covariant derivative of the Ricci tensor $Ric^{\nabla}$ are given by  
\begin{eqnarray*}
(\nabla_{\partial_1} Ric^{\nabla})(\partial_1,\partial_1)&=& -\partial_1 b 
+ 2bf^{1}_{11} + f^{2}_{11}(d-a);\\
(\nabla_{\partial_1} Ric^{\nabla})(\partial_1,\partial_2)&=& -\partial_1 d
+ d(f^{1}_{11} + f^{2}_{12}) + bf^{1}_{12} - cf^{2}_{11};\\
(\nabla_{\partial_1} Ric^{\nabla})(\partial_2,\partial_1)&=& \partial_1 a
- a(f^{1}_{11} + f^{2}_{12}) + bf^{1}_{12} - cf^{2}_{11};\\
(\nabla_{\partial_1} Ric^{\nabla})(\partial_2,\partial_2)&=& \partial_1 c
+ (d-a)f^{1}_{12} -2cf^{2}_{12};\\
(\nabla_{\partial_2} Ric^{\nabla})(\partial_1,\partial_1)&=& -\partial_2 b 
+ 2bf^{1}_{12} + f^{2}_{12}(d-a);\\
(\nabla_{\partial_2} Ric^{\nabla})(\partial_1,\partial_2)&=& -\partial_2 d
+ d(f^{1}_{12} + f^{2}_{22}) + bf^{1}_{22} - cf^{2}_{12};\\
(\nabla_{\partial_2} Ric^{\nabla})(\partial_2,\partial_1)&=& \partial_2 a
- a(f^{1}_{12} + f^{2}_{22}) + bf^{1}_{22} - cf^{2}_{12};\\
(\nabla_{\partial_2} Ric^{\nabla})(\partial_2,\partial_2)&=& \partial_2 c
+ (d-a)f^{2}_{22} -2cf^{2}_{22}.
\end{eqnarray*}
This completes the proof.
\end{proof}

\textit{Case 2:}  
Let $M$ be a three-dimensional smooth manifold and $\nabla$ be an affine torsion-free connection. Suppose that the action of the affine connection $\nabla$ on the basis of the tangent space $\{\partial_{i}\}_{1\le i\le 3}$ is given by  
\begin{equation}\label{3DimNabla1}
 \nabla_{\partial_{i}} \partial_{i }= f_{i} \partial_{i}, \;\;\mbox{for}\;\; i=1,2,3,
\end{equation}
where $f_{i} =f_{i}(u_1,u_2,u_3)$ are smooth functions. Then the non-zero components of the curvature tensor $\mathcal{R}^{\nabla}$ of the connection $\nabla$ defined in (\ref{3DimNabla1}) are given by
\begin{equation}
 \mathcal{R}^{\nabla} (\partial_{i},\partial_{j})\partial_{i} =  -\partial_{j} f_{i} \partial_{i} \;\;\mbox{and}\;\; \mathcal{R}^{\nabla} (\partial_{i},\partial_{j})\partial_{j} = \partial_{i} f_{j} \partial_{j}, 
\end{equation}
for $i\neq j, i, j=1,2,3$. The non-zero components of the Ricci tensor of the connection (\ref{3DimNabla1}) are given by 
\begin{equation}
 Ric^{\nabla} (\partial_{i},\partial_{j})=-\partial_{i}f_{j},\;\;\mbox{for}\;\;i\neq j, i, j=1,2,3.
\end{equation}
The non-zero components of the covariant derivative of the Ricci tensor are given by 
\begin{align}
(\nabla_{\partial_{i}} Ric^{\nabla})(\partial_{j},\partial_{k})&= -\partial_{i}\partial_{j}f_{k} ,\;\;\;
 (\nabla_{\partial_{i}} Ric^{\nabla})(\partial_{i},\partial_{j}) = -\partial_{i}^{2}f_{j} + f_{i}\partial_{i}f_{j},\nonumber\\
 (\nabla_{\partial_{i}} Ric^{\nabla})(\partial_{j},\partial_{i})&= -\partial_{i}\partial_{j}f_{i} + f_{i}\partial_{j}f_{i}, 
\end{align}
for $i\neq j\neq k, i, j, k=1,2,3$. In this case, the relation (\ref{RicciCycP2}) is explicitly given by 
\begin{align} 
 (\nabla_X Ric^{\nabla})(X,X) & = \alpha_{1}^{2}\alpha_{2}\{-\partial_{1}^{2}f_{2} + f_{1}\partial_{1} f_{2} - \partial_{1}\partial_{2} f_{1} + f_{1}\partial_{2} f_{1}\}  \nonumber\\
 &+\alpha_{1}^{2}\alpha_{3}\{ -\partial_{1}^{2}f_{3} + f_{1}\partial_{1} f_{3} - \partial_{1}\partial_{3} f_{1} + f_{1}\partial_{3} f_{1}\}\nonumber\\
 &+\alpha_{1} \alpha_{2}^{2}\{ -\partial_{2}^{2}f_{1}-\partial_{2}\partial_{1}f_{2} + f_{2}\partial_{1} f_{2}  + f_{2}\partial_{2} f_{1}\}\nonumber\\
 &+\alpha_{1}\alpha_{3}^{2}\{ -\partial_{3}^{2}f_{1} + f_{3}\partial_{3} f_{1} - \partial_{3}\partial_{1} f_{3} + f_{3}\partial_{1} f_{3}\}\nonumber\\
 &+\alpha_{2}^{2}\alpha_{3}\{ -\partial_{2}^{2}f_{3} + f_{2}\partial_{2} f_{3} - \partial_{2}\partial_{3} f_{2} + f_{2}\partial_{3} f_{2}\}\nonumber\\
 &+\alpha_{2}\alpha_{3}^{2}\{ -\partial_{3}^{2}f_{2} + f_{3}\partial_{3} f_{2} - \partial_{3}\partial_{2} f_{3} + f_{3}\partial_{2} f_{3}\}\nonumber\\
 &-2\alpha_{1}\alpha_{2}\alpha_{3}\{ \partial_{1}\partial_{2}f_{3}+\partial_{1}\partial_{3}f_{2}+\partial_{2}\partial_{3}f_{1} \},\nonumber
\end{align}
for $\displaystyle X=\sum_{i=1}^{3}\alpha_{i}\partial_{i}$. Therefore, we have the following result.
\begin{proposition}\label{RiccCylPA2} 
The affine connection $\nabla$ defined in (\ref{3DimNabla1}) satisfies  (\ref{RicciCycP1}) if the functions $f_{i}$, for $i=1,2, 3$, satisfy the following partial differential equations:
\begin{align}
 \partial_{(i}\partial_{j}f_{k)}=0,  
 \;\;\mbox{and} \;\; \partial_{i}^{2}f_{j}+\partial_{i}\partial_{j} f_{i} - f_{i}(\partial_{i} f_{j}+  \partial_{j} f_{i} )=0, 
\end{align}
for $i\neq j\neq k, i, j, k=1,2,3$.
\end{proposition}
The manifolds with cyclic parallel Ricci tensor, known as $L_{3}$-spaces, are well-developed in Riemannian geometry (see \cite{ks} and \cite{pt}, and references therein). The cyclic parallelism of the Ricci tensor is sometimes called the ``\textit{First Ledger condition}'' \cite{pt}. In \cite{sz2}, for instance, the author proved that a smooth Riemannian manifold satisfying the first Ledger condition is real analytic. Tod in \cite{to} used the same condition to characterize the four-dimensional K\"ahler manifolds which are not Einstein. It has also enriched the D'Atri spaces (see \cite{ks} and \cite{pt}, for more details).

\section{The affine Szab\'o manifolds}\label{AffineMani}

In this section we adapt the definition of pseudo-Riemannian Szab\'o manifold given by Fiedler and Gilkey in \cite{fg} to the affine case. We shall prove that, on a smooth affine surface, the affine Szab\'o condition is closely related to the cyclic parallelism of the Ricci tensor.

\begin{definition}{\rm
Let $(M,\nabla)$ be a smooth affine manifold and $p\in M$. 
\begin{enumerate}
\item[(i)] $(M,\nabla)$ is called affine Szab\'o at $p \in M$ if the affine Szab\'o operator $\mathcal{S}^{\nabla}(X)$ has the same characteristic polynomial for every vector field $X$ on $M$.
\item[(ii)]  $(M,\nabla)$ is called affine Szab\'o if $(M,\nabla)$ is affine Szab\'o at each 
$p\in M$.
\end{enumerate}
}
\end{definition}
\begin{theorem}\label{TheorSZA1}
Let $(M,\nabla)$ be an $n$-dimensional affine manifold and $p\in M$. Then $(M,\nabla)$ is affine Szab\'o at $p\in M$ 
if and only if the characteristic polynomial of the affine Szab\'o operator $\mathcal{S}^{\nabla}(X)$ is
$
 P_{\lambda}(\mathcal{S}^{\nabla}(X)) = \lambda^{n},
$
for every $X \in T_p M$.
\end{theorem}
\begin{proof}
If the characteristic polynomial of the affine Szab\'o operator at $p$ is given by $P_{\lambda}(\mathcal{S}^{\nabla}(X)) = \lambda^{n}$, then the affine manifold $(M,\nabla)$ is obviously affine Szab\'o. 
Assume that $(M,\nabla)$ is affine Szab\'o, then for $X\in T_p M$, the characteristic polynomial of the affine Szab\'o operator $\mathcal{S}^{\nabla}(X)$ is given by
$
P_{\lambda} [\mathcal{S}^{\nabla}(X)]= \lambda^{n}-\sigma_{1}\lambda^{n-1} + \sigma_{2}\lambda^{n-2}-\cdots + (-1)^{n}\sigma_{n}.
$
Then for $\beta\in \mathbb{R}$, $\beta\neq 0$, the characteristic polynomial of the affine Szab\'o operator
$\mathcal{S}^{\nabla}(\beta X)$ is given by
$
P_{\lambda} [\mathcal{S}^{\nabla}(\beta X)]= \lambda^{n}-\sigma_{1} \beta^{3}\lambda^{n-1} + \sigma_{2}\lambda^{n-2}-\cdots + (-1)^{n} \beta^{3n}\sigma_{n}.
$
Since $(M,\nabla)$ is affine Szab\'o, that is $P_{\lambda} [\mathcal{S}^{\nabla}(X)] =P_{\lambda} [\mathcal{S}^{\nabla}(\beta X)]$, it follows that $\sigma_{1} = \cdots = \sigma_{n}=0$ which complete the proof.
\end{proof}
\begin{corollary}
If $(M,\nabla)$ is affine Szab\'o at $p\in M$, then the Ricci tensor of $(M,\nabla)$ is cyclic parallel.
\end{corollary} 
Now, we give a complete description of affine Szab\'o surfaces.  We shall prove the following result:
\begin{theorem}\label{main}
Let $(M,\nabla)$ be a two-dimensional smooth affine manifold. Then $(M,\nabla)$ is affine Szab\'o at $p\in M$ if and only 
if the Ricci tensor of $(M,\nabla)$ is cyclic parallel at $p\in M$.
\end{theorem}
\begin{proof}
Suppose that $(M,\nabla)$ is affine Szab\'o. Let $X= \alpha_i \partial_i, \,  i=1,2$ be a vector on $M$, then, using the connection (\ref{CoefCon2}), the affine Szab\'o operator is given by
\begin{equation}
(\nabla_{X} \mathcal{R}^{\nabla}) (\partial_1,X)X =  A \partial_1 + B \partial_2, \;\;
(\nabla_{X} \mathcal{R}^{\nabla}) (\partial_2,X)X = C \partial_1 + D \partial_2,\nonumber
\end{equation}
where the coefficients $A$, $B$, $C$ and $D$ are given by
\begin{align*}
A &=  \alpha^{2}_{1} \alpha_2 [\partial_1 a - a (f^{1}_{11} + f^{2}_{12})
+ b f^{1}_{12} - cf^{2}_{11}]\\ 
&+  \alpha_1\alpha^{2}_{2} [\partial_2 a + \partial_1 c - a (f^{1}_{12} + f^{2}_{22})
+ (d-a)f_{12}^{1}+ bf^{1}_{22} -3cf^{2}_{12}]\\
&+  \alpha^{3}_{2} [\partial_2 c -2cf^{2}_{22} + (d-a)f^{1}_{22}], \\
B &=  \alpha^{2}_{1} \alpha_2 [\partial_1 b -2bf^{1}_{11} - (d-a)f^{2}_{11}]\\
&+  \alpha_1\alpha^{2}_{2} [\partial_2 b + \partial_1 d - 3bf^{1}_{12} +cf^{2}_{11} 
-(d-a)f^{2}_{12} - d(f_{11}^{1} + f_{12}^{2})]\\
&+  \alpha^{3}_{2}[\partial_2 d -bf^{1}_{22} + cf^{2}_{12} - d(f^{1}_{12}+f^{2}_{22})], 
\end{align*}
\begin{align*} 
C &=  \alpha^{3}_{1} [-\partial_1 a + a(f^{1}_{11} +f^{2}_{12}) -bf^{1}_{12}]\\
&+  \alpha^{2}_{1}\alpha_2 [-\partial_2 a -\partial_1 c + a (f^{1}_{12} +f^{2}_{22})
-bf^{1}_{22} +3cf^{2}_{12} -(d-a)f^{1}_{12}]\\
&+  \alpha_1\alpha^{2}_{2} [-\partial_2 c + 2cf^{2}_{22} - (d-a)f^{1}_{22}],\\
D &=  \alpha^{3}_{1}[-\partial_1 b +2 bf^{1}_{11} + (d-a)f^{2}_{11}]\\
&+  \alpha^{2}_{1}\alpha_2 [-\partial_2 b - \partial_1 d + 3bf^{1}_{12} - cf^{2}_{11}
+ d(f^{1}_{11} + f^{2}_{12}) + (d-a)f^{2}_{12}]\\
&+  \alpha_1\alpha^{2}_{2} [-\partial_2 d + bf^{1}_{22} -cf^{2}_{12} 
+ d(f^{1}_{12} + f^{2}_{22})].
\end{align*}
The matrix associated to $\mathcal{S}^{\nabla} (X)$ with respect to the basis 
$\{\partial_1, \partial_2\}$ is given by
\begin{eqnarray*}
(\mathcal{S}^{\nabla} (X)) = 
\left(\begin{array}{cc}
       A&B\\
       C&D\\
       \end{array}
\right).
\end{eqnarray*}
Its characteristic polynomial is given by   
$
 P_{\lambda} [\mathcal{S}^{\nabla} (X)]=\lambda^2 -\lambda(A+D) + (AD-BC).
$
Since $(M,\nabla)$ is affine Szab\'o, by Theorem \ref{TheorSZA1}, $0$ is the only eigenvalue of the affine Szab\'o operator $\mathcal{S}^{\nabla}(X)$. Therefore, $\det(\mathcal{S}^{\nabla}(X))= AD-BC=0$ and $trace(\mathcal{S}^{\nabla}(X))= A+D=0$. The latter implies that  
\begin{align}
&\partial_2 c -2cf^{2}_{22} + (d-a)f^{1}_{22} =  0,\;\;\;
 -\partial_1 b +2 bf^{1}_{11} + (d-a)f^{2}_{11}  = 0,\nonumber\\
&\partial_1 a -\partial_2 b - \partial_1 d + 4bf^{1}_{12} - 2cf^{2}_{11} 
+ (d-a)(f^{1}_{11} + 2f^{2}_{12})  = 0,\nonumber\\
&\partial_2 a + \partial_1 c -\partial_2 d + 2bf^{1}_{22} -4cf^{2}_{12}
+(d-a)(2f^{1}_{12} + f^{2}_{22})  = 0.\nonumber
\end{align}
 The converse is obvious.
\end{proof}
\begin{corollary} 
Let $\nabla$ be the affine connection on $\mathbb{R}^2$ defined by 
$\nabla_{\partial_1}\partial_1 = f^{1}_{11} \partial_1$, $\nabla_{\partial_1} \partial_2 = 0$, $\nabla_{\partial_2} \partial_2 = f^{2}_{22}\partial_2$. Then $\nabla$ is affine Szab\'o if and only if the functions $f^{1}_{11}=f^{1}_{11} (u_1,u_2) ,f^{2}_{22}= f^{2}_{22} (u_1,u_2)$ satisfy the following partial differential equations: $\partial_{1} a - \partial_{1}d + (d-a) f^{1}_{11}   = 0$, $\partial_{2} a - \partial_{2}d + (d-a) f^{2}_{22}  = 0$, where $a$ and $d$ are defined in (\ref{Smalabcd}).
\end{corollary}
To support this, we have the following example. Consider on $\mathbb{R}^{2}$ the torsion-free connection $\nabla$ with the only non-zero coefficient functions given by 
$$
\nabla_{\partial_1} \partial_1 = (u_{1}+u_{2})\partial_{1}\;\;\mbox{and}\;\;\nabla_{\partial_2} \partial_2 = (u_{1}+u_{2}+1)\partial_{2}.
$$
It is easy to check that $(\mathbb{R}^{2}, \nabla)$ is an affine Szab\'o manifold.
\begin{corollary} 
Let $\nabla$ be the affine connection on $\mathbb{R}^2$ defined by $\nabla_{\partial_1} \partial_1 = 0$, $\nabla_{\partial_1} \partial_2 = f^{1}_{12} \partial_1$, $\nabla_{\partial_2} \partial_2 = f^{1}_{22}\partial_1$. Then $\nabla$ is affine Szab\'o if and only if the functions $f^{1}_{12}= f^{1}_{12} (u_1,u_2)$ and $f^{1}_{22}=f^{1}_{22} (u_1,u_2)$ satisfy the following partial differential equations: $\partial_{1} a  = 0$, $\partial_{2} c -a f^{1}_{22}  =0$, $\partial_{2} a + \partial_{1} c -2af^{1}_{12}  = 0$, where $a$ and $c$ are defined in (\ref{Smalabcd}).
\end{corollary}
Let consider on $\mathbb{R}^{2}$ the torsion-free connection $\nabla$ with the only non-zero coefficient functions given by 
$ 
\nabla_{\partial_1} \partial_2 = u_{2}\partial_{1}\;\;\mbox{and}\;\;\nabla_{\partial_2} \partial_2 = u_{1}(1+u_{2})\partial_{1}.
$ 
It is easy to check that $(\mathbb{R}^{2}, \nabla)$ is an affine Szab\'o manifold.

Now, we give an example of a family of affine Szab\'o connections on a $3$-dimensional manifold. Let us 
consider the affine connection defined in (\ref{3DimNabla1}), i.e., 
\begin{equation} 
 \nabla_{\partial_{i}} \partial_{i }= f_{i} \partial_{i}, \;\;\mbox{for}\;\; i=1,2,3,\nonumber
\end{equation}
where $f_{i} =f_{i}(u_{1}, u_{2}, u_{3})$ are smooth functions. For $\displaystyle X=\sum_{i=1}^{3}\alpha_{i}\partial_{i}$, the affine Szab\'o operator is given by
$$
(\nabla_{X} \mathcal{R}^{\nabla}) (\partial_{i},X)X =  \sum_{j=1}^{3} A_{ji}\partial_{j},
$$
where
\begin{align*}
 A_{11} &= \alpha^{2}_{1}\alpha_2 (-\partial_1\partial_2 f_1 + f_1\partial_2 f_1)
 + \alpha^{2}_{1}\alpha_3 (-\partial_1\partial_3 f_1 + f_1\partial_3 f_1)\\
 &+ \alpha_{1}\alpha^{2}_{2} (-\partial^{2}_{2} f_1 + f_2\partial_2 f_1)
 + \alpha_{1}\alpha^{2}_{3} (-\partial^{2}_{3} f_1 + f_3\partial_3 f_1)\\
 &+  \alpha_1\alpha_2\alpha_3(-2\partial_2\partial_3 f_1),\\ 
 A_{21} &=   \alpha^{3}_{2}(\partial_2\partial_1 f_2 - f_2\partial_1 f_2)
 + \alpha_{1}\alpha^{2}_{2}(\partial^{2}_{1} f_2 - f_1\partial_1 f_2)
 + \alpha^{2}_{2}\alpha_{3} (\partial_3\partial_1 f_2),\\
 A_{31} &=   \alpha^{3}_{3}(\partial_3\partial_1 f_3 - f_3\partial_1 f_3)
 + \alpha_{1}\alpha^{2}_{3}(\partial^{2}_{1} f_3 - f_1\partial_1 f_3)
 + \alpha_{2}\alpha^{2}_{3} (\partial_2\partial_1 f_3),\\
 A_{12} &=  \alpha^{3}_{1}(\partial_1\partial_2 f_1 - f_1\partial_2 f_1)
 + \alpha^{2}_{1}\alpha_{2}(\partial^{2}_{2} f_1 - f_2\partial_2 f_1)
 + \alpha^{2}_{1}\alpha_{3} (\partial_3\partial_2 f_1),\\
 A_{22} &=   \alpha^{2}_{1}\alpha_2 (-\partial^{2}_{1} f_2 + f_1\partial_1 f_2)
 + \alpha_{1}\alpha^{2}_{2} (-\partial_2\partial_1 f_2 + f_2\partial_1 f_2)\\
 &+  \alpha^{2}_{2}\alpha_{3} (-\partial_{2}\partial_3 f_2 + f_2\partial_3 f_2)
 + \alpha_{2}\alpha^{2}_{3} (-\partial^{2}_{3} f_2 + f_3\partial_3 f_2)\\
 &+  \alpha_1\alpha_2\alpha_3(-2\partial_1\partial_3 f_2), \\
 A_{32} &=   \alpha^{3}_{3}(\partial_3\partial_2 f_3 - f_3\partial_2 f_3)
 + \alpha_{2}\alpha^{2}_{3}(\partial^{2}_{2} f_3 - f_2\partial_2 f_3)
 + \alpha_{1}\alpha^{2}_{3} (\partial_1\partial_2 f_3), \\ 
 A_{13} &=  \alpha^{3}_{1}(\partial_1\partial_3 f_1 - f_1\partial_3 f_1)
 + \alpha^{2}_{1}\alpha_{3}(\partial^{2}_{3} f_1 - f_3\partial_3 f_1)
 + \alpha^{2}_{1}\alpha_{2} (\partial_2\partial_3 f_1), \\
 A_{23} &=  \alpha^{3}_{2}(\partial_2\partial_3 f_2 - f_2\partial_3 f_2)
 + \alpha^{2}_{2}\alpha_{3}(\partial^{2}_{3} f_2 - f_3\partial_3 f_2)
 + \alpha_{1}\alpha^{2}_{2} (\partial_1\partial_3 f_2),\\
 A_{33} &=   \alpha^{2}_{1}\alpha_3 (-\partial^{2}_{1} f_3 + f_1\partial_1 f_3)
 + \alpha_{1}\alpha^{2}_{3} (-\partial_3\partial_1 f_3 + f_3\partial_1 f_3)\\
 &+ \alpha^{2}_{2}\alpha_{3} (-\partial^{2}_{2} f_3 + f_2\partial_2 f_3)
 + \alpha_{2}\alpha^{2}_{3}(-\partial_{3}\partial_2 f_3 + f_3\partial_2 f_3)\\
 &+  \alpha_1\alpha_2\alpha_3(-2\partial_1\partial_2 f_3).
\end{align*}
For specific functions $f_{i}$, we have the following.
\begin{theorem}
Let $M=\mathbb{R}^3$ and let $\nabla$ be the torsion-free connection, whose the non-zero coefficients of the connection are given by $ f_1 = \frac{1}{2}u_1u^{2}_{2}$, $f_2 =-\frac{1}{2}u^{2}_{1}u_2$ and $f_3 = u_3$. Then $(M,\nabla)$ is an affine Szab\'o manifold.
\end{theorem}
We have also the following family of examples of affine Szab\'o connections.
\begin{theorem} 
Let us consider a torsion free connection on $\mathbb{R}^3$ given by the following: $\nabla_{\partial_i} \partial_k = \frac{1}{u_i}\partial_k$, $\nabla_{\partial_j} \partial_k = \frac{1}{u_j}\partial_k$, $\nabla_{\partial_i} \partial_j  =  \frac{u_k}{u_iu_j}\partial_k$, 
with $i\neq j \neq k; i,j,k=1,2,3$ and $u_i\neq 0,u_j\neq 0, u_k\neq 0$. Then
$(M,\nabla)$ is affine Szab\'o.
\end{theorem} 
\begin{proof}
It is easy to see that the curvature tensor of the affine connections is flat.
\end{proof}
\begin{example}{\rm
The following affine connections on $\mathbb{R}^3$ given by:
\begin{enumerate}
\item $\nabla_{\partial_1} \partial_2 = \frac{1}{u_2}\partial_1,\,
\nabla_{\partial_1} \partial_3 = \frac{1}{u_3}\partial_1,\,
\nabla_{\partial_2} \partial_3 = \frac{u_1}{u_2u_3}\partial_1 $;
\item $\nabla_{\partial_1} \partial_2 = \frac{1}{u_1}\partial_2,\,
\nabla_{\partial_1} \partial_3 = \frac{u_2}{u_1u_3}\partial_2,\,
\nabla_{\partial_2} \partial_3 = \frac{1}{u_3}\partial_2$;
\item $\nabla_{\partial_1} \partial_2 = \frac{u_3}{u_1u_2}\partial_3,\quad
\nabla_{\partial_1} \partial_3 = \frac{1}{u_1}\partial_3,\quad
\nabla_{\partial_2} \partial_3 = \frac{1}{u_2}\partial_3 $;
\end{enumerate}
are affine Szab\'o.}
\end{example}

\section{A classification of locally homogeneous affine Szab\'o manifolds in dimension two}

Homogeneity is one of the fundamental notions in differential geometry. In this section we consider the homogeneity of manifolds with affine connections in dimension two. This homogeneity means that for each two points of a manifold there is an affine transformation which sends one point into another. We characterize locally homogeneous connections which are Szab\'o in a two dimensional smooth manifold. Note that Locally homogeneous Riemannian structures were first studied by Singer in \cite{si}.

A smooth connection $\nabla$ on $M$ is \textit{locally homogeneous} \cite{op} if and only if it admits, in neighborhoods of each point $p \in M$, at least two linearly independent affine Killing vectors fields. An affine Killing vector field $X$ is characterized by the equation:
\begin{eqnarray}\label{locallyhomogeneous}
 [X,\nabla_Y Z] - \nabla_Y [X,Z] - \nabla_{[X,Y]} Z =0
\end{eqnarray}
for any arbitrary vectors fields $Y$ and $Z$ on $M$. Let us express the vector field $X$ on $M$ in the form 
\begin{equation*}
X = F(u_1,u_2) \partial_1 + G(u_1,u_2) \partial_2.
\end{equation*}
Writing the formula (\ref{locallyhomogeneous}) in local coordinates, we find that any affine Killing vector field $X$ must satisfy six basics equations. We shall write these equations in the simplified notation: 
\begin{align*}
\partial_{11} F + f^{1}_{11} \partial_1 F +\partial_1 f^{1}_{11} F - f^{2}_{11} \partial_2 F
+ \partial_2 f^{1}_{11} G + 2 f^{1}_{12} \partial_1 G &=  0,\\
\partial_{11} G + 2f^{2}_{11} \partial_1 F + (2f^{2}_{12} -f^{1}_{11}) \partial_1 G
-f^{2}_{11} \partial_2 G + \partial_1 f^{2}_{11} F + \partial_2 f^{2}_{11} G &=  0,\\
\partial_{12} F + (f^{1}_{11} - f^{2}_{12})\partial_2 F + f^{1}_{22} \partial_1 G
+ f^{1}_{12} \partial_2 G + \partial_1 f^{1}_{12} F + \partial_2 f^{1}_{12} G &=  0,\\
\partial_{12} G + f^{2}_{12} \partial_1 F + f^{2}_{11} \partial_2 F 
+ (f^{2}_{22} -f^{2}_{11}) \partial_1 G +\partial_1 f^{2}_{12} F + \partial_2 f^{2}_{12} G &=  0,\\
\partial_{22} F - f^{1}_{22}\partial_1 F +(2f^{1}_{12} - f^{2}_{22})\partial_2 F
+ 2f^{1}_{22}\partial_2 G +\partial_1 f^{1}_{22} F + \partial_2 f^{1}_{22} G &=  0,\\
\partial_{22} G +2 f^{2}_{12}\partial_2 F - f^{1}_{22}\partial_1 G + f^{2}_{22})\partial_2 G
\partial_1 f^{2}_{22} F + \partial_2 f^{2}_{22} G &=  0.
\end{align*}
The following result is the first classification of torsion free homogeneous connections on two dimensional manifolds.
\begin{theorem}\label{theoremak}\cite{op}
Let $\nabla$ be a locally homogeneous torsion free affine connection on a two-dimensional manifold $M$. Then, in a
neighborhood $\mathcal{U}$ of each point $u\in M$, either $\nabla$ is the Levi-Civita connection of the 
standard metric of the unit sphere or, there is a system $(u_1,u_2)$ of local coordinates and constants $a,b,c,d,e,f$
such that $\nabla$ is expressed in $\mathcal{U}$ by one of the following formulas:
\begin{enumerate}
 \item Type $\mathcal{A}$:
 \begin{eqnarray*}
  \nabla_{\partial_1}\partial_1 = a\partial_1 + b\partial_2,\;\;  
 \nabla_{\partial_1}\partial_2 = c\partial_1 + d\partial_2,\;\; 
\nabla_{\partial_2}\partial_2 = e\partial_1 + f\partial_2.
 \end{eqnarray*}
\item Type $\mathcal{B}$:
\begin{eqnarray*}
  \nabla_{\partial_1}\partial_1 = \frac{1}{u_1}(a\partial_1 + b\partial_2),\;\; 
 \nabla_{\partial_1}\partial_2 = \frac{1}{u_1}(c\partial_1 + d\partial_2),\;\; 
\nabla_{\partial_2}\partial_2 = \frac{1}{u_1}(e\partial_1 + f\partial_2).
 \end{eqnarray*}
\end{enumerate}
\end{theorem}
Next, we characterize all affine connections given in Theorem \ref{theoremak} which are \textit{affine Szab\'o}.
\begin{theorem}\label{TheoTypeA}
 The affine manifolds of type $\mathcal{A}$ are affine Szab\'o if and only if they have parallel Ricci tensor.
\end{theorem}
\begin{proof}
The components of the Ricci tensor are given by $ Ric(\partial_1,\partial_1)  =  (ad -d^2 +bf -bc)$, $ Ric(\partial_1,\partial_2)  =  (cd-be)$, $Ric(\partial_2,\partial_1)  =  (cd-be)$, $ Ric(\partial_2,\partial_2)  =  (ae -de +cf-c^2)$. 
The Ricci tensor is symmetric. Then, the covariant derivatives of the Ricci tensor are given by
\begin{eqnarray*}
 (\nabla_{\partial_1} Ric)(\partial_1,\partial_1) &=& 2(abc +ad^2 -a^2d -abf +b^2e -bcd),\\
 (\nabla_{\partial_1} Ric)(\partial_1,\partial_2) &=& 2(bc^2 +bde -acd -bcf),\\
 (\nabla_{\partial_1} Ric)(\partial_2,\partial_2) &=& 2(bce -ade -cdf +d^2e),\\
 (\nabla_{\partial_2} Ric)(\partial_1,\partial_1) &=& 2(bc^2 +bde -acd -bcf ),\\
 (\nabla_{\partial_2} Ric)(\partial_1,\partial_2) &=& 2(bce -ade -cdf +d^2e),\\
 (\nabla_{\partial_2} Ric)(\partial_2,\partial_2) &=& 2(be^2  +c^2f -cf^2 -aef -cde +def).
\end{eqnarray*}
From Theoreom \ref{main}, the proof is complete.
\end{proof}
\begin{theorem}\label{TheoTypeB}
 The affine manifolds of type $\mathcal{B}$ are affine Szab\'o if and only if the coefficients $a,b,c,d,e$ and $f$ satisfy
 \begin{align*}
  2abc+3bc-d-2ad-a^2d-bcd+d^2+ad^2+b^2e &=  0,\\
  2c+ac+4bc^2-2cd-3acd+3be+3bde+2bce &= 0,\\
  3c^2+3c^2d+e-ae+3bce+2de-3ade+3d^2e &= 0,\\
  -2c^3+ace-2cde+be^2 &= 0.
 \end{align*}
\end{theorem}
\begin{proof}
 The components of the Ricci tensor are given by
 \begin{align*}
 Ric(\partial_1,\partial_1) &= \frac{1}{u^{2}_{1}}[d+ d(a -d) +b(f -c)],\;\;
 Ric(\partial_1,\partial_2)  =  \frac{1}{u^{2}_{1}} (f + cd - be),\\
 Ric(\partial_2,\partial_1) &=  \frac{1}{u^{2}_{1}} (-c + cd - be),\;\;
 Ric(\partial_2,\partial_2)  =  \frac{1}{u^{2}_{1}}[-e + e(a -d) + c(f -c)],
 \end{align*}
 and it is symmetric if and only if $f=-c$ holds. So we set $f=-c$. Then, the covariant derivatives of the
 Ricci tensor are given by
 \begin{align*}
  (\nabla_{\partial_1} Ric)(\partial_1,\partial_1) &=  \frac{2}{u^{3}_{1}} (2abc+3bc-d-2ad-a^2d-bcd+d^2+ad^2+b^2e),\\
  (\nabla_{\partial_1} Ric)(\partial_1,\partial_2) &=  \frac{1}{u^{3}_{1}} (2c+ac+2bc^2-2cd-2acd+3be+2bde+2bce),\\
  (\nabla_{\partial_1} Ric)(\partial_2,\partial_2) &=  \frac{2}{u^{3}_{1}} (3c^2+c^2d+e-ae+bce+2de-ade+d^2e),\\
  (\nabla_{\partial_2} Ric)(\partial_1,\partial_1) &=  \frac{2}{u^{3}_{1}} (2bc^2-acd+bde),
   \end{align*}
   \begin{align*}
  (\nabla_{\partial_2} Ric)(\partial_1,\partial_2) &=  \frac{2}{u^{3}_{1}} (c^2d+bce-ade+d^2e),\\
  (\nabla_{\partial_2} Ric)(\partial_2,\partial_2) &=  \frac{2}{u^{3}_{1}} (-2c^3+ace-2cde+be^2).
 \end{align*}
A straightforward calculation using the Theorem \ref{main} completes the proof.
\end{proof} 
As stated by Brozos \textit{et al.} in \cite{broz2},  the surfaces of Type $\mathcal{A}$ and Type $\mathcal{B}$ can have quite different geometric properties. For instance, the Ricci tensor of any Type $\mathcal{A}$ surface is symmetric while this property can fail for a Type $\mathcal{B}$ surface. Thus the geometry of a Type $\mathcal{B}$ surface is not as rigid as that of a Type $\mathcal{A}$ surface. This is closely related to the existence of non-flat affine Osserman structures (see \cite{gar} and many references therein). This difference in terms of geometric properties is also remarkable when those surfaces satisfy the Szab\'o condition (Theorem  \ref{TheoTypeA} and \ref{TheoTypeB}).

In the paper \cite{broz1}, the authors determined the moduli space of Type $\mathcal{A}$ affine geometries. Depending on the signature it is either a smooth
2-dimensional surface or a smooth 2-dimensional surface with a single cusp point (signature (2, 0)). They also wrote down complete sets of invariants that determine the local isomorphism type depending on the rank of the Ricci tensor. 

Clearly the condition that the Szab\'o operator is nilpotent is gauge invariant and therefore depends only on the Christoffel symbols modulo the action of the gauge group. This opens perspective studies in order to have more invariant formulation using recent classification results of Brozos-Vazquez \textit{et al.} (see \cite{broz1, broz2} for more details). Note that the classification of locally homogeneous affine connections in two dimension is a nontrivial problem. (see \cite{ak} and \cite{op} and the reference therein for more information).

\section{The twisted Riemannian extensions of an affine Szab\'o manifold}

Affine Szab\'o connections are of interest not only in affine geometry, but also in the study of pseudo-Riemannian Szab\'o metrics since they provide some nice examples without Riemannian analogue by means of the Riemannian extensions and the twisted Riemannian extensions. 

A pseudo-Riemannian manifold $(M,g)$ is said to be Szab\'o if the Szab\'o operators $(\nabla_X R)(\cdot,X)X$
has constant eigenvalues on the unit pseudo-sphere bundles $S^{\pm}(TM)$ (\cite{gis}). Any Szab\'o manifold is locally 
symmetric in the Riemannian ~\cite{sz1} and the Lorentzian ~\cite{gs} setting but the higher signature case supports examples 
with nilpotent Szabó operators (cf. ~\cite{gis} and the references therein). Next we will use the twisted
Riemannian contruction to exhibit a four-dimensional Szab\'o metric where the degree of nilpotency of the associated
Szab\'o operators changes at each point depending on the direction.

Let $(M,\nabla)$ be an affine manifold of dimension $n$. The \textit{Riemannian extension} is the pseudo-Riemannian metric 
$g_{\nabla}$ on $T^* M$ of neutral signature $(n,n)$, which is given in local coordinates relative to the frame 
$\{\partial_{u_1},\cdots,\partial_{u_n},\partial_{u_{1'}},\cdots,\partial_{u_{n'}}\}$ by
\begin{equation}
 g_{\nabla} = 2du_i \circ du_{i'} - 2u_{k'}\Gamma^{k}_{ij}du_i \circ du_j
\end{equation}
where $\Gamma^{k}_{ij}$ give the Christoffel symbols of the affine connection $\nabla$. Riemannian extension were 
originally defined by Patterson and Walker \cite{pw} and further investigated in relating pseudo-Riemannian properties 
of $N$ with the affine structure of the base manifold $(M,\nabla)$. Moreover, Riemannian extension were also considered 
in \cite{ks} in relation with $L_3$-spaces. One has:
\begin{theorem}
Let $(M,\nabla)$ be a two-dimensional smooth torsion-free affine manifold. Then the following assertions are equivalent:
\begin{enumerate}
 \item $(M,\nabla)$ is an affine Szab\'o manifold.
 \item The Riemannian extension $(T^*M,g_{\nabla})$ of $(M,\nabla)$ is a pseudo-Riemannian nilpotent Szab\'o manifold
 of neutral signature.
\end{enumerate}
\end{theorem}
We also have the following:
\begin{theorem}\cite{ks}
Let $(M,\nabla)$ be a smooth torsion-free affine manifold of dimension $n\geq 3$. Then the following
assertions hold:
\begin{enumerate}
\item If $(M,\nabla)$ is an affine Szab\'o manifold, then its Riemannian extension $(T^*M, g_{\nabla})$ is a pseudo-Riemannian Szab\'o manifold.
\item If the Ricci tensor of $(M,\nabla)$ is symmetric and the Riemannian extension $(T^*M,g_{\nabla})$ of $(M,\nabla)$
is a pseudo-Riemannian Szab\'o manifold, then $(M,\nabla)$ is an affine Szab\'o manifold.
\end{enumerate}
\end{theorem}

More generally, if $\Phi$ is a symmetric $(0,2)$-tensor field on $M$, then the \textit{twisted Riemannian extension} 
$g_{\nabla,\Phi}$, is the metric of neutral signature on $T^*M$ given by
\begin{equation}
 g_{\nabla,\Phi} = 
\left(\begin{array}{cc}
       \Phi_{ij}(\vec{u})-2u_{k'}\Gamma^{k}_{ij}&Id_n\\
       Id_n&0\\
       \end{array}
\right).
\end{equation}
Thus in particular, if $\nabla$ is flat, the Szab\'o operators of $g_{\nabla,\Phi}$ are nilpotent and the couple $(N,g_{\nabla,\Phi})$ is a Szab\'o pseudo-Riemannian manifold (\cite{broz}). Here, we consider the twisted Riemannian of a not flat affine connection and we will prove the following result:
\begin{theorem}
 Let $M=\mathbb{R}^2$ and let $\nabla$ be the torsion-fres connection, whose the non-zero Christoffel symbols are
 given by $\nabla_{\partial_1} \partial_1 = (u_{1}+u_{2})\partial_{1}$ and $\nabla_{\partial_2} \partial_2 = (u_{1}+u_{2}+1)\partial_{2}$. 
Let $\overline{g}:=g_{\nabla,\Phi}$ on $T^*M$. Then $\overline{g}$ is a Szab\'o metric of signature $(2,2)$. Moreover
the degree of nilpotency of the Szab\'o operators $(\nabla_X R)(\cdot,X)X$ depends on the direction $X$ at each point. 
\end{theorem}
\begin{proof}
Let $(M,\nabla)$ be a $2$-dimensional affine manifold. The twisted Riemannian extension of the following connection
$
\nabla_{\partial_1} \partial_1 = (u_{1}+u_{2})\partial_{1}
$
and 
$
\nabla_{\partial_2} \partial_2 = (u_{1}+u_{2}+1)\partial_{2}
$
is the pseudo-Riemannian metric $\overline{g}$ on the cotangent bundle $T^* M$ of neutral signature $(2,2)$ defined by 
\begin{eqnarray*}
\overline{g} &=& \Big[\Phi_{11}(u_1,u_2)-2(u_1+u_2)u_3\Big]du_1\otimes du_1 
+ 2\Phi_{12}(u_1,u_2)du_1\otimes du_2\\ 
&+& 2du_1\otimes du_3 + \Big[\Phi_{22}(u_1,u_2)-2(u_1+u_2+1)u_4\Big]du_2\otimes du_2
+ 2du_2\otimes du_4.
\end{eqnarray*}
The Levi-Civita connection is determined by the Christoffel symbols as follows:
\begin{eqnarray*}
 \varGamma^{1}_{11} &=& u_1+u_2,\quad \varGamma^{2}_{22} = (u_1+u_2+1),\\
 \varGamma^{3}_{13} &=& -(u_1+u_2),\quad \varGamma^{4}_{24} = -(u_1+u_2+1),\\
 \varGamma^{3}_{11}&=& \frac{1}{2}\partial_1 \Phi_{11}(u_1,u_2) - (u_1+u_2)[\Phi_{11}(u_1,u_2)-2(u_1+u_2)u_3]-u_3,\\
 \varGamma^{4}_{11} &=& \partial_1 \Phi_{12}(u_1,u_2) -\frac{1}{2}\partial_2 \Phi_{11}(u_1,u_2) -(u_1+u_2)\Phi_{12}(u_1,u_2)+u_3,\\
 \varGamma^{3}_{12} &=& \frac{1}{2}\partial_2 \Phi_{11}(u_1,u_2) -u_3,\quad
 \varGamma^{4}_{12} = \frac{1}{2}\partial_1 \Phi_{22}(u_1,u_2) -u_4,\\
 \varGamma^{3}_{22} &=& \partial_2 \Phi_{12}(u_1,u_2) -\frac{1}{2}\partial_1 \Phi_{22}(u_1,u_2)-(u_1+u_2+1)\Phi_{12}(u_1,u_2)+u_4,\\
 \varGamma^{4}_{22}&=& \frac{1}{2}\partial_2\Phi_{22}(u_1,u_2)-(u_1+u_2+1)[\Phi_{22}(u_1,u_2)-2(u_1+u_2+1)u_4]-u_4.
\end{eqnarray*}
A straightforward calculation from the Christoffel symbols shows that the non zero components curvature tensor are given by
\begin{align*}
 R(\partial_1,\partial_2)\partial_1 &=  -\partial_1 + \Big[\Phi_{11} 
 -2(u_1+u_2)u_3 \Big]\partial_3\\
 &+ \Big[\frac{1}{2}\partial^{2}_{1}\Phi_{22} -\partial_1\partial_2 \Phi_{12} +\frac{1}{2}\partial^{2}_{2}\Phi_{11}
 +\Phi_{12}\\
 &+  (u_1+u_2+1)\Big(\partial_1\Phi_{12}-\frac{1}{2}\partial_2\Phi_{11}\Big)\\ 
 &+ (u_1+u_2)\Big(\partial_2\Phi_{12}-\frac{1}{2}\partial_1\Phi_{22}\Big)-(u_1+u_2)(u_1+u_2+1)\Phi_{12}\\
 &+  (u_1+u_2+1)u_3+(u_1+u_2)u_4\Big]\partial_4,
 \end{align*}
 \begin{align*} 
 R(\partial_1,\partial_2)\partial_2 &=  \partial_2 -\Big[\frac{1}{2}\partial^{2}_{1}\Phi_{22}-\partial_1\partial_2 \Phi_{12}
 +\frac{1}{2}\partial^{2}_{2}\Phi_{11}+ \Phi_{12}\\ 
 &+  (u_1+u_2+1)\Big(\partial_1\Phi_{12}-\frac{1}{2}\partial_2\Phi_{11}\Big)\\
 &+ (u_1+u_2)\Big(\partial_2\Phi_{12}-\frac{1}{2}\partial_1\Phi_{22}\Big)-(u_1+u_2)(u_1+u_2+1)\Phi_{12}\\
 &+  (u_1+u_2+1)u_3+(u_1+u_2)u_4\Big]\partial_3\\
 &-  \Big[\Phi_{22}-2(u_1+u_2+1)u_4\Big]\partial_4,\\
 R(\partial_1,\partial_2)\partial_3 &=  \partial_3, \;\;
 R(\partial_1,\partial_2)\partial_4 =-\partial_4, \;\;
 R(\partial_1,\partial_3)\partial_1 =-\partial_4, \;\;
 R(\partial_1,\partial_3)\partial_2 = \partial_3.
\end{align*}
Let $X=\sum_{i=1}^{4}\alpha_i\partial_i$ be a non-null vector, where $\{\partial_i\}$ denotes the coordinates basis. The associated Szab\'o operator $(\nabla_X R)(\cdot,X)X$ can be expressed with respect to the coordinates basis $\{\partial_i\}$ 
as follows:
\begin{eqnarray}
\mathcal{S}(X)= 
\left(\begin{array}{cccc}
       a_{11}&a_{12}&0&0\\
       a_{21}&a_{22}&0&0\\
       a_{31}&a_{32}&a_{33}&0\\
       a_{41}&a_{42}&a_{43}&0\\
       \end{array}
\right),
\end{eqnarray}
with
\begin{align*}
 a_{11} &= f_1(u_1,u_2,u_3,u_4),\, a_{21} = f_2(u_1,u_2,u_3,u_4),\,a_{31} = f_3(u_1,u_2,u_3,u_4);\\ 
 a_{41} &= f_4(u_1,u_2,u_3,u_4),\, a_{21} = f_1(u_1,u_2,u_3,u_4),\,a_{22} = f_2(u_1,u_2,u_3,u_4);\\
 a_{32} &= f_3(u_1,u_2,u_3,u_4),\, a_{42} = f_4(u_1,u_2,u_3,u_4);\\
 a_{33} &= [\alpha^{2}_{1}\alpha_2(u_1+u_2) + \alpha_{1}\alpha^{2}_{2}(u_1+u_2+1);\\
 a_{43} &= -[\alpha^{3}_{1}(u_1+u_2) + \alpha^{2}_{1}\alpha_2(u_1+u_2+1).
\end{align*}
For the particular choice of the unit vectors $X_1=\partial_1 + \partial_3$ and $X_2=\partial_2+\partial_4$, respectively, it is easy to shows that $\mathcal{S}(X_1)$ is three-step nilpotent while $\mathcal{S}(X_2)$ is two-step nilpotent.
\end{proof}

\section*{Acknowledgments}

The first author would like to thank  the University of KwaZulu-Natal for financial support. The authors would like to thank Professor P. Gilkey (University of Oregon, USA) for reading the manuscript and for his valuable comments. They also thank the referee for his/her valuable suggestions and comments.


\begin{thebibliography}{xxx} 
\bibitem{ak} T. Arias-Marco and O. Kowalski, Classification of locally homogeneous affine connections with arbitrary torsion on 2-dimensional manifolds, Monatsh. Math., 153 (2008), (1), 1-18.
\bibitem{broz1} M. Brozos-V\'azquez, E. Garc\'ia-Rio and P. Gilkey, Homogeneous affine surfaces: Affine Killing vector fields and gradient Ricci-solitons,  http://arxiv.org/abs/1512.05515. 
\bibitem{broz2} M. Brozos-V\'azquez, E. Garc\'ia-Rio and P. Gilkey, Homogeneous affine surfaces: Moduli spaces, , http://arxiv.org/abs/1604.06610.
\bibitem{broz} M. Brozos-V\'azquez, E. Garc\'ia-Rio, P. Gilkey, S. Nikevi\'c and R. V\'azquez-Lorenzo, The Geometry of Walker Manifolds, Synthesis Lectures on Mathematics and Statistics, 5. Morgan and Claypool Publishers, Williston, VT, 2009.
\bibitem{di} A. S. Diallo, Affine Osserman connections on 2-dimensional manifolds, Afr. Diaspora J. Math. 11 (1) (2011), 103-109.
\bibitem{fg} B. Fiedler and P. Gilkey,  Nilpotent Szabo, Osserman and Ivanov-Petrova pseudo-Riemannian manifolds, Contemp. Math., 337 (2003),53-64.
\bibitem{gar} E. Garc\'ia-Rio, D. N. Kupeli, M. E. V\'azquez-Abal and R. V\'azquez-Lorenzo, Affine Osserman connections and their Riemannian extensions, Differential Geom. Appl. 11 (1999), 145-153.
\bibitem{gi} P. B. Gilkey, Geometric properties of natural operators defined by the Riemannian curvature tensor, World Scientific Publishing, ISBN 981-02-4752-4.
\bibitem{gis} P. B. Gilkey, R. Ivanova and I. Stavrov, Jordan Szabó algebraic covariant derivative curvature tensors. Recent advances in Riemannian and Lorentzian geometries, Contemp. Math., 337 (2003), 65-75. 
\bibitem{giz} P. B. Gilkey, R. Ivanova and T. Zhang, Szab\'o Osserman IP Pseudo-Riemannian manifolds, Publ. Math. Debrecen, 62 (2003),387-401.
\bibitem{gs} P. Gilkey and I. Stavrov, Curvature tensors whose Jacobi or Szab\'o operator is nilpotent on null vectors,  Bull. London Math. Soc. 34 (2002), no. 6, 650-658.
\bibitem{ks} O. Kowalski and M. Sekizawa, The Riemann extensions with cyclic parallel Ricci tensor, Math. Nachr. 287, No 8-9(2014), 955-961.
\bibitem{ns} K. Nomizu and T. Sasaki, Affine Differential Geometry. Geometry of Affine Immersions. Cambridge Tracts in Mathe-
matics Vol. 111 (Cambridge University Press, Cambridge, 1994).
\bibitem{op} B. Opozda, A classification of locally homogeneous connections on $2$-dimensional manifolds, Differential Geom. Appl., 21 (2004), (2), 87-102.
\bibitem{pw} E. P. Patterson and A. G. Walker, Riemann extensions, Quart. J. Math. Oxford Ser. (2) 3 (1952), 19-28.
\bibitem{pt} H. Pedersen and P. Tod, The Ledger curvature conditions and D'Atri geometry, Differential Geom. Appl. 11 (1999), no. 2, 155-162.
\bibitem{si} I.M. Singer, Infinitesimally homogeneous spaces, Comm. Pure Appl. Math. 13 (1960) 685-697.
\bibitem{sz1} Z. I. Szab\'o, A short topological proof for the symmetry of $2$ point homogeneous spaces, Invent. Math., 106 (1991), 61-64.
\bibitem{sz2} Z. I. Szab\'o, Spectral theory for operator families on Riemannian manifolds. Differential geometry: Riemannian geometry (Los Angeles, CA, 1990), 615-665, Proc. Sympos. Pure Math., 54, Part 3, Amer. Math. Soc., Providence, RI, 1993.
\bibitem{to} K. P. Tod, Four-dimensional D'Atri Einstein spaces are locally symmetric. Differential Geom. Appl. 11 (1999), no. 1, 55-67.
\end{thebibliography}
\end{document}